\documentclass[11pt]{article}
\usepackage{epsfig}
 \usepackage{psfrag}
\usepackage{amsmath}
\usepackage{amsfonts} \makeatother
\usepackage{amsfonts,amsmath,amssymb,mathrsfs}
\usepackage{cite}
\usepackage{amsthm}
\usepackage{float}
\usepackage{tikz}
\usepackage{graphicx}
\usepackage{verbatim}
\usepackage{indentfirst}
\numberwithin{equation}{section}
\newtheorem{theorem}{Theorem}[section]

\newtheorem{lemma}[theorem]{Lemma}

\newcommand{\be}{\begin{equation}}
\newcommand{\ee}{\end{equation}}
\newcommand\bes{\begin{eqnarray}} \newcommand\ees{\end{eqnarray}}
\newcommand{\bess}{\begin{eqnarray*}}
\newcommand{\eess}{\end{eqnarray*}}

\begin{document}

\title{Accurate gradient computations at interfaces using finite element methods}

\author{Fangfang Qin
\thanks{Jiangsu Key Laboratory for NSLSCS, School of Mathematical Sciences,
Nanjing Normal University, Nanjing 210023, China}
\and Zhaohui Wang\thanks{Department of Mathematics, North Carolina State University, Raleigh, NC 27695}
\and Zhijie Ma\thanks{China Institute of Water Resource and Hydropower Research (IWHR), Beijing, China}
\and Zhilin Li\thanks{Department of Mathematics, North Carolina State University, Raleigh, NC 27695, USA and Nanjing Normal University. zhilin@math.ncsu.edu} \footnote{Corresponding author.} }

\maketitle

\setlength{\baselineskip}{16pt}
\begin{abstract}
New finite element  methods  are  proposed for elliptic interface problems in one and two dimensions.
The main motivation is not only to get an accurate solution
but also an accurate first order derivative at the interface (from each side).
The key in 1D is to use the idea from \cite{wheeler1974galerkin}.
For 2D interface problems, the idea is to introduce a small tube near the interface and introduce the gradient as part of unknowns,
which is similar to a mixed finite element method, except only at  the interface.
Thus the computational cost is just slightly higher than the standard finite element method.
We present rigorous one dimensional analysis,
which show second order convergence order for both of the solution and the gradient in 1D.
For two dimensional problems,
we present  numerical results and  observe   second order convergence  for the solution,
and super-convergence for the gradient at the interface.
\end{abstract}



\section{Introduction}

In this paper, we consider  the following  interface problems
\begin{equation} \label{nabla-pde ch2}
 -  \nabla \cdot (\beta(\mathbf{x})  \nabla u(\mathbf{x})) + q(\mathbf{x}) u(\mathbf{x}) = f(\mathbf{x}), \qquad \mathbf{x}\in \Omega\setminus \Gamma,
\end{equation}
in one and two dimensions.
We assume that there is a closed interface $\Gamma$ in the solution domain across
which the coefficient $\beta$ has a finite jump (discontinuity)
\begin{equation}\label{beta ch2}
\beta(\mathbf{x})  =
\begin{cases}
 \beta_1  & \mbox{if $\mathbf{x} \in \Omega_1,$}\\
 \beta_2  & \mbox{if $\mathbf{x} \in \Omega_2.$}
 \end{cases}
\end{equation}
Because of the discontinuity,
the natural jump condition should be satisfied,
that is, both of the solution and the flux should be continuous across the interface $\Gamma$
\begin{equation}\label{flux ch2}
 [u]_{\Gamma}=0, \qquad \left[\beta \frac{\partial u}{\partial n} \right]_{\Gamma} = 0,
\end{equation}
where the jump at a point $\mathbf{X}=(X,Y)$ on the interface $\Gamma$ is defined as
\begin{small}
\begin{equation*}
\left[\beta \frac{\partial u}{\partial n} \right]_{\mathbf{X}}
  = \lim_{\mathbf{x} \rightarrow\mathbf{X}, \mathbf{x} \in \Omega_2} \beta(\mathbf{x}) \frac{\partial u(\mathbf{x})}{\partial n}  - \lim_{\mathbf{x}\rightarrow \mathbf{X}, \mathbf{x}\in \Omega_1} \beta(\mathbf{x}) \frac{\partial u(\mathbf{x})}{\partial n},
\end{equation*}
\end{small}
and  $u_n = \mathbf{n} \cdot \nabla u=\frac{\partial u}{\partial n}$ is the normal derivative of solution $u(\mathbf{X})$,
and $\mathbf{n}$ is the unit normal direction pointing to $\Omega_2$ side,
see Fig.~\ref{fig:domain ch2} for an illustration.
Since a finite element discretization is used,
we assume that $f(\mathbf{x})\in L^2(\Omega_i)$,  $q(\mathbf{x})\in L^{\infty}(\Omega_i)$ excluding $\Gamma$.
For the regularity requirement of the problem,
we also assume that $\beta(\mathbf{x})\ge \beta_0 >0$ and $q(\mathbf{x})\ge 0$; $\Gamma\in C^1$.
From these assumptions,
we know that the solution $u(\mathbf{x})\in H^2(\Omega_i)$ for $i=1,2$.
\begin{figure}[!b]
 \centering
  \includegraphics[width=0.45\textwidth]{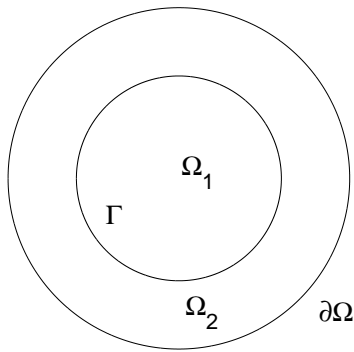}
  \caption{A diagram of domain $\Omega_1$, $\Omega_2$, an interface $\Gamma$, and the boundary $\partial\Omega$. }\label{fig:domain ch2}
\end{figure}

There are many  applications of such an interface problem, see for example,
\cite{sutton1995interfaces, zienkiewicz2000finite, li2006immersed}
and the references therein.
Many numerical methods have been developed for solving such an important problem.
For the elliptic interface problem (\ref{nabla-pde ch2})-(\ref{flux ch2}),
the solution has a low global regularity.
Thus, a direct conforming finite element method based on polynomial basis functions over a mesh likely works poorly
if the mesh  is not aligned  along the interface since the FE solution will be a smooth piece in an element
and can not capture the discontinuity in the directive at the interface.
Nevertheless,  it is reasonable to assume that the solution is piecewise smooth excluding the interface.
For example, if the coefficient is a piecewise constant in each sub-domain,
then the solution in each sub-domain is an analytic function in the interior,
but has jump in the normal derivative of the solution across the interface from the PDE limiting theory \cite{kevorkian1990partial}.
The gradient used in this paper is defined as the \emph{limiting gradient} from each side of the interface.

Naturally, finite element methods can be and have been applied to solve interface problems.
It is well known that a second order accurate approximation to the solution
of an interface problem  can be generated by the Galerkin finite element method
with the standard linear basis functions
if the triangulation is aligned with the interface,
that is, a body fitted mesh is used, see for example,
\cite{babuska1970finite, bramble1996finite, chen1998finite, xu1982error}.
Other state of art methods include the IGA-FEM,
or DPG, 
a discontinuous Petrov-Galerkin finite element method \cite{carstensena2014low}.
Some kind of posterior techniques or at least quadratic elements are needed
in order to get second order accurate gradient from each side of the interface.
The cost in the mesh generation coupled with unstructured linear solver
makes the body-fitted mesh approach less competitive.

Alternatively, we can use a structured mesh (non-body fitted) to solve such an interface problem.
There are also quite a few finite element methods using Cartesian meshes.
The immersed finite element method (IFEM) was developed for 1D and 2D interface problems
in \cite{li1998immersed} and \cite{li2003new}, respectively.
Since then, many IFEM methods and analysis have appeared in the literature,
see for example, \cite{chou2010optimal, he2010immersed},
with applications in \cite{lin2012linear, yang2003immersed}.
Often they provide second order accurate solution in $L^2$ norm but only first order accurate flux.

Nevertheless, in many applications the primary interest
is focused on flux values at interfaces
in addition to solutions of governing differential equations,
see for example, \cite{chou2012immersed, douglas1974galerkin, wheeler1974galerkin}.
Most of numerical methods for interface problems based on structure meshes
are between first and second order accurate for the solution
but the accuracy for the gradient is usually one order lower.
Note that the gradient recovering techniques for examples, \cite{wahlbin1995superconvergence, zhang2005new},
hardly work for structured meshes because of the arbitraries of the interface and the underlying mesh.
The mixed finite element approach and a few other methods
that can find accurate solutions and  gradients simultaneously
in the entire domain are often lead to saddle problems
and are computationally expensive
which are not ideal choices
if we are only interested in an accurate gradient near an interface or a boundary.
In this paper, we develop two new finite element methods, one is in 1D, the other one is 2D,
for obtaining  accurate approximations of the flux values at interfaces.

The rest of the paper is organized as follows.
In Section 2,
we explain the one dimensional algorithm and provide the theoretical analysis.
We explain how to construct approximations to the flux values at the left and the right of the interface,
and approximations to the flux values at the boundary of the domain.
The numerical algorithm for two dimensional problems is explained in Section 3 along with some numerical experiments.
We conclude and acknowledge in Section 4.

\section{One-dimensional algorithm and analysis}

The one dimensional model problem has the following form,
\begin{equation}\label{sol ch2}
\begin{split}
& -\left(\beta\left(x\right) u^{'}(x)\right)^{'}+q\left(x\right) u\left(x\right)=f\left(x\right),\\
&\qquad \qquad u\left(0\right)=0,\qquad u\left(1\right)=0,
\end{split}
\end{equation}
where $0<x<1$,
$\beta$ is a piecewise  constant and have a finite jump  at an interface $0<\alpha<1$,
and homogenous boundary condition for simplicity of the discussion.
Across the interface,  the natural jump conditions are:
\begin{equation} \label{jumpcond ch2}
 \left[ u\right]_{\alpha} =0,  \qquad  \left[\beta  u^{'}(x)\right]_{\alpha}=0.
\end{equation}

We define the standard bilinear form,
\begin{equation*}
a\left(u,v\right)=\int_{0}^{1}
\left(\beta\left(x\right)u^{'}\left(x\right)v^{'}\left(x\right)+q\left(x\right)u\left(x\right)v\left(x\right)\right) dx,\quad\forall~ u\left(x\right), v\left(x\right) \in H^{1}_{0}\left(0,1\right),
\end{equation*}
where $H^{1}_{0}\left(0,1\right)$ is the Sobolev space
\cite{brenner2007mathematical, adams2003sobolev, tartar2007introduction}.
\begin{equation*}
 H^{1}_{0}\left(0,1\right)=\{v\left(x\right) \in H^{1} \left(0,1\right)~~\mbox{and}~~v(0)=v(1)=0 \}.
\end{equation*}
The solution of the differential equation $u\left(x\right)\in H^{1}_{0}\left(0,1\right)$
is also the solution of the following variational problem:
\begin{equation}\label{varform ch2}
a\left(u,v\right)=\left(f,v\right)
=\int_{0}^{1} f\left(x\right)v\left(x\right) dx, \qquad  \forall~ v\left(x\right) \in H^{1}_{0}\left(0,1\right).
\end{equation}

Integration by parts over the separated intervals $ \left(0,\alpha\right)$ and $\left(\alpha,1\right)$ yields,
\begin{equation*}\label{}
0=\int_{0}^{\alpha}\left\{-\left(\beta u^{'}\right)^{'}+qu-f\right\}vdx+\beta_{1}u^{-}_{x}v^{-}+\int_{\alpha}^{1}\left\{-\left(\beta u^{'}\right)^{'}+qu-f\right\}vdx-\beta_{2}u^{+}_{x}v^{+}.
\end{equation*}
The superscripts $-$ and $+$ indicate the limiting value as $x$ approaches $\alpha$ from the left and right, respectively, and $u_{x}=u^{'}$. Recall that $v^{-}=v^{+}$ for any $v$ in $H^{1}_{0}$, it follows that the differential equation holds in each interval and that
\begin{equation*}\label{}
\left[u\right]=u^{+}-u^{-}=0,\quad\left[\beta u_{x}\right]=\beta^{+}u^{+}_{x}-\beta^{-}u^{-}_{x}=0,
\end{equation*}
where we have dropped the subscript $\alpha$ in the jumps for simplicity.
These relations are the same as in $\left(\ref{sol ch2}\right)$,
which indicates that the discontinuity in the coefficient $\beta\left(x\right)$
does not cause any additional difficulty for the theoretical  analysis of the finite element method.
The weak solution will satisfy the jump conditions (\ref{jumpcond ch2}).

\subsection{The immersed finite element method  in 1D}

Now we briefly explain the immersed finite element method (IFEM) in 1D introduced in \cite{li1998immersed}.
As in the IFEM, we use a uniform grid, i.e., $x_{i}=ih, i=0, 1,\cdots, n, $
and assume that $\alpha \in \left(x_{j}, x_{j+1}\right)$.
Since it is one dimensional problem,
 we use $\beta^+$ for $\beta_2$, $\beta^-$ for $\beta_1$ and so on.

If the interface does not cut an interval $(x_i,x_{i+1})$,
then we use the standard piecewise linear basis function, \emph{i.e.},
the hat function $\phi_i(x)$, $i=1,2, \cdots$ if $i\neq j$ and $i\neq j+1$.
For $x_{j}$ and $x_{j+1}$, the associated piecewise linear basis functions $\phi_j(x)$, $\phi_{j+1}(x)$ are modified.
For example, $\phi_j(x)$ is defined as a piecewise linear function that satisfies
\begin{equation*}\label{}
 \varphi_j(x_j) = 1, \quad \varphi_j(x_i)=0, \quad   \left[\varphi_{j}\right]=0,\quad\left[\beta \varphi^{'}_{j}\right]=0.
\end{equation*}
It has been derived in \cite{li2006immersed} that
\begin{equation*}\label{}
\varphi_{j}\left(x\right)=
\begin{cases}
 ~0  ,\qquad \qquad \qquad~~~~~0\leq x <x_{j-1},\\
   ~\frac{x-x_{j-1}}{h} ,\qquad \quad~~~~~~x_{j-1}\leq x < x_{j},\\
  ~ \frac{x_{j}-x}{D}+1 ,\qquad ~~~~~~~x_{j}\leq x < \alpha,\\
 ~ \frac{\rho \left(x_{j+1}-x\right)}{D} ,\qquad ~~~~~~\alpha\leq x < x_{j+1},\\
 ~0  ,\qquad \qquad \qquad~~~~~x_{j+1}\leq x \leq 1,
\end{cases}
\end{equation*}
where
\begin{equation*}\label{}
\rho =\frac{\beta_{1}}{\beta_{2}},\quad D=h-\frac{\beta_{2}-\beta_{1}}{\beta_{2}}\left(x_{j+1}-\alpha\right),
\end{equation*}
and
\begin{equation*}\label{}
\varphi_{j+1}\left(x\right)=
\begin{cases}
 ~0  ,\qquad \qquad \qquad~~~~~0\leq x <x_{j},\\
   ~\frac{x-x_{j}}{D} ,\qquad \quad\quad~~~~~~x_{j}\leq x < \alpha,\\
  ~ \frac{\rho \left(x-x_{j+1}\right)}{D}+1 , ~~~~~~\alpha\leq x < x_{j+1},\\
  ~ \frac{x_{j+2}-x}{h}  ,\qquad \quad~~~~~~x_{j+1}\leq x < x_{j+2},\\
  ~0  ,\qquad \qquad \qquad~~~~~x_{j+2}\leq x \leq 1.
\end{cases}
\end{equation*}

Let $V_{h, \left(0, 1\right)}\triangleq$ Span $\left\{\varphi_{i}\right\}^{n-1}_{i=1}$
be the immersed finite element space for approximating $u$.
We propose the following bilinear form for problem
$\left(\ref{sol ch2}\right)$:
find $u_{h} \in V_{h, \left(0, 1\right)} \subset H^{1}_{0}\left(0,1\right)$ such that
\begin{equation}\label{c ch2}
a\left(u_{h},v_{h}\right)=\left(f,v_{h}\right),\qquad\forall~ v_{h} \in V_{h, \left(0, 1\right)},
\end{equation}

\subsection{Error analysis for 1D IFEM}

Some error analysis of 1D IFEM has been given in  \cite{li1998immersed, li2006immersed}.
Here we provide somewhat different and more traditional analysis.
As usual, we study the approximation property of the IFE space $V_{h, \left(0, 1\right)}$
so that we can bound the error of the finite element solution using that of the interpolation function.

Assuming that $x_{j}\leq \alpha <x_{j+1}$,
we define the interpolation operator $\pi_{h}: H^{1}_{0}\left(0,1\right)\longrightarrow V_{h, \left(0, 1\right)}$ as follows:
\begin{equation}\label{int ch2}
\pi_{h} u\left(x\right)=
\begin{cases}
 ~\frac{x_{i+1}-x}{h}u\left(x_{i}\right)+\frac{x-x_{i}}{h}u\left(x_{i+1}\right),  \quad x_{i}< x < x_{i+1},~i\neq j,\\
 ~\kappa\left(x-x_{j}\right)+u\left(x_{j}\right),\qquad\qquad ~x_{j}< x\leq \alpha,\\
 ~\kappa \frac{\beta_{1}}{\beta_{2}}\left(x-x_{j+1}\right)+u\left(x_{j+1}\right),  \quad \alpha\leq x <x_{j+1},
 \end{cases}
\end{equation}
where
\begin{equation*}\label{}
\kappa=\frac{\beta_{2}\left(u\left(x_{j+1}\right)-u\left(x_{j}\right)\right)}{\beta_{2}
\left(\alpha-x_{j}\right)-\beta_{1}\left(\alpha-x_{j+1}\right)}.
\end{equation*}
It is easy to verify that $\pi_{h} u\left(x_{i}\right)=u\left(x_{i}\right), i=0, 1,\cdots, n$, $\left[\pi_{h} u\right]=0$, and $\left[\beta\pi_{h} u^{'}\right]=0$.

Now we pay attention to the estimation of $\left\|u\left(x\right)-\pi_{h} u\left(x\right)\right\|_{0}$.
Here, we define $E_{i}=\left[x_{i}, x_{i+1}\right], i\neq j$ and $I=\left[x_{j}, x_{j+1}\right]$.
For a regular element, we use the classical finite element method to  estimate  the error.
\begin{equation}\label{pp ch2}
\left\|u-\pi_{h} u\right\|_{1, E_{i}}\leq ch\left\|u\right\|_{2, E_{i}}.
\end{equation}
We are going  to focus on the error analysis for the element which contains the interface.
We first define
\begin{small}
\begin{equation*}
\widetilde{H}^{2}\left(0, 1\right)=\left\{v\in H^{1}_{0}\left(0, 1\right),\\
~v\in H^{2}\left(0, \alpha\right),~v\in H^{2}\left(\alpha, 1\right)\right\}.
\end{equation*}
\end{small}
equipped with the norm and the semi-norm,
\begin{equation*}
\|u\|^{2}_{\widetilde{H}^{2}\left(0, 1\right)} \triangleq  \|u\|^{2}_{H^{2}\left(0, \alpha\right)}+\|u\|^{2}_{H^{2}\left(\alpha, 1\right)} \end{equation*}
and
\begin{equation*}
|u|^{2}_{\widetilde{H}^{2}\left(0, 1\right)} \triangleq |u|^{2}_{H^{2}\left(0, \alpha\right)}+|u|^{2}_{H^{2}\left(\alpha, 1\right)}.
\end{equation*}
Then we have the following error estimate for the derivative approximation, $\kappa\sim u_x^-$ \ from left.

\begin{lemma}{}
If $u\left(x\right)$ is the solution of $\left(\ref{sol ch2}\right)$, the following inequality holds:
\begin{equation}\label{lem ch2}
\begin{split}
&\left|u_{x}^{-}\left(\alpha\right)-\kappa\right|\leq ch^{\frac{1}{2}}\left\|u\right\|_{2,I},\\
&\left|u_{x}^{+}\left(\alpha\right)-\kappa\rho\right|\leq ch^{\frac{1}{2}}\left\|u\right\|_{2,I}.
\end{split}
\end{equation}
\end{lemma}
\begin{proof}
Using the Taylor expansion at $\alpha$ and the jump conditions $\left(i.e.,  \left(\ref{jumpcond ch2}\right)\right)$, we have
\begin{eqnarray*}
&&\left|u_{x}^{-}\left(\alpha\right)-\kappa\right|=\left|u_{x}^{-}\left(\alpha\right)-\frac{\beta_{2}\left(u\left(x_{j+1}\right)-u\left(x_{j}\right)\right)}{\beta_{2}\left(\alpha-x_{j}\right)-\beta_{1}\left(\alpha-x_{j+1}\right)}\right|\\
&&~~~~\quad~\qquad~=\bigg|u_{x}^{-}\left(\alpha\right)-\frac{\beta_{2}\left\{u^{+}\left(\alpha\right)+u^{+}_{x}\left(\alpha\right)\left(x_{j+1}-\alpha\right)+\int_\alpha^{x_{j+1}} u^{''}\left(t\right)\left(x_{j+1}-t\right)dt\right\}}{\beta_{2}\left(\alpha-x_{j}\right)-\beta_{1}\left(\alpha-x_{j+1}\right)}\\
&&~~\qquad\qquad\quad~~~-\frac{u^{-}\left(\alpha\right)+u^{-}_{x}\left(\alpha\right)\left(x_{j}-\alpha\right)+\int_\alpha^{x_{j}}u^{''}\left(t\right)\left(x_{j}-t\right)dt}{\beta_{2}\left(\alpha-x_{j}\right)-\beta_{1}\left(\alpha-x_{j+1}\right)}\bigg|\\
&&~~~\quad~\qquad~~=\left|\frac{\beta_{2}\left[\int_\alpha^{x_{j+1}} u^{''}\left(t\right)\left(x_{j+1}-t\right)dt-\int_\alpha^{x_{j}} u^{''}\left(t\right)\left(x_{j}-t\right)dt\right]}{\beta_{2}\left(\alpha-x_{j}\right)-\beta_{1}\left(\alpha-x_{j+1}\right)}\right| 
\\
&&~~~\quad~\qquad~~\leq ch^{\frac{1}{2}}\left\|u\right\|_{2,I}, \qquad\qquad\qquad~~~~\qquad(\rm{Cauchy-Schwarz ~~ Inequality})
\end{eqnarray*}
where $c$ is a positive constant depending only on the coefficients $\beta$, $q(x)$.
This completes the proof of the lemma.
\end{proof}

The lemma gives rough estimates of the first order derivative of the interpolation function
from each side of the interface with an $O(\sqrt{h})$ convergence order
compared with that $O(h)$ in the $H^1(0,1)$ of the interpolation function.
Later on, we will explain our method to get  second order accurate derivative
from each side of the interface.

In a similar way, we can prove that
\begin{equation*}
\left|u_{x}^{+}\left(\alpha\right)-\kappa\rho\right|\leq ch^{\frac{1}{2}}\left\|u\right\|_{2,I}.
\end{equation*}

\subsection{Convergence analysis of 1D IFEM}

Although some error analysis is given in \cite{li1998immersed},
below we provide some different,
more traditional finite element analysis with some results
that are useful for accurate gradient computations at the interface.
First we prove the following theorem on the accuracy of the interpolating function~$\pi_{h} u$.
\begin{theorem}{}
If $u\left(x\right)$ is the solution of $\left(\ref{sol ch2}\right)$,
and $\pi_{h}u\left(x\right)$ is the interpolation function defined in $\left(\ref{int ch2}\right)$, then
\begin{equation}\label{theo ch2}
\left\|u-\pi_{h} u\right\|_{1, I}\leq ch\left\|u\right\|_{2,I},
\end{equation}
where $c$ is a positive constant depending on the interface location,
the coefficients $\beta$, and $q(x)$.
\end{theorem}
\begin{proof}
Proof of theorem If $x_{j}\leq x\leq \alpha$,
then using the Taylor expansion, we have
\begin{eqnarray*}
&&\qquad\qquad~~u(x)=u\left(x_{j}\right)+u^{'}\left(x_{j}\right)\left(x-x_{j}\right)+\int_{x_{j}}^x u^{''}\left(t\right)\left(x-t\right)dt\\
&&~~~\qquad\qquad\quad~~=u\left(x_{j}\right)+\left[u^{-}_{x}\left(\alpha\right)+\int_{\alpha}^{x_{j}} u^{''}\left(x\right)dx\right]\left(x-x_{j}\right)+\int_{x_{j}}^x u^{''}\left(t\right)\left(x-t\right)dt\\
&&~~~\qquad\qquad\quad~~=u\left(x_{j}\right)+u^{-}_{x}\left(\alpha\right)\left(x-x_{j}\right)+\int_{\alpha}^{x_{j}} u^{''}\left(x\right)dx\left(x-x_{j}\right)+\int_{x_{j}}^x u^{''}\left(t\right)\left(x-t\right)dt.
\end{eqnarray*}
By $\left(\ref{lem ch2}\right)$ and the Cauchy-Schwartz inequality, we get
\begin{eqnarray*}
&&\left|u(x)-\pi_{h} u\left(x\right)\right|=\left|\left(u_{x}^{-}\left(\alpha\right)-k\right)\left(x-x_{j}\right)+\int_{\alpha}^{x_{j}} u^{''}\left(x\right)dx\left(x-x_{j}\right)+\int_{x_{j}}^{x} u^{''}\left(t\right)\left(x-t\right)dt\right|\\
&&~~~\qquad\qquad\quad~~\leq c h^{\frac{3}{2}}\left\|u\right\|_{2,I},
\end{eqnarray*}
and furthermore
\begin{eqnarray*}
&&\left|\left(u\left(x\right)-\pi_{h} u\left(x\right)\right)^{'}\right|=\left|u_{x}^{-}\left(\alpha\right)-k+\int_{\alpha}^{x_{j}}u^{''}\left(x\right)dx+\int_{x_{j}}^{x}u^{''}\left(t\right)dt\right|\\
&&\quad~~~~\qquad\qquad\quad~~\leq ch^{\frac{1}{2}}\left\|u\right\|_{2,I}.
\end{eqnarray*}

If $\alpha \leq x\leq x_{j+1}$, the proof is similar. Thus we also have,
\begin{equation*}
\left|u(x)-\pi_{h} u\left(x\right)\right|\leq c h^{\frac{3}{2}}\left\|u\right\|_{2,I},   \qquad\left|\left(u\left(x\right)-\pi_{h} u\left(x\right)\right)^{'}\right|\leq ch^{\frac{1}{2}}\left\|u\right\|_{2,I}.
\end{equation*}

We proceed with the remaining proof below
\begin{eqnarray*}
&&\left\|u\left(x\right)-\pi_{h} u\left(x\right)\right\|_{0, I}=\left(\int_{x_{j}}^{\alpha}\left(u\left(x\right)-\pi_{h} u\left(x\right)\right)^{2}dx+\int_{\alpha}^{x_{j+1}}\left(u\left(x\right)-\pi_{h} u\left(x\right)\right)^{2}dx\right)^{\frac{1}{2}}\\
&&~~\quad~~~\qquad\qquad\quad~~\leq c\left(\int_{x_{j}}^{\alpha}h^{3}\left\|u\right\|_{2,I}^{2}dx+\int_{\alpha}^{x_{j+1}}h^{3}\left\|u\right\|_{2,I}^{2}dx\right)^{\frac{1}{2}} \\
&&~~\quad~~~\qquad\qquad\quad~~\leq ch^{2}\left\|u\right\|_{2,I},
\end{eqnarray*}
in $L^2$ and we continue to $H^1$,
\begin{eqnarray*}
&&\left|u\left(x\right)-\pi_{h} u\left(x\right)\right|_{1, I}=\left(\int_{x_{j}}^{\alpha}\left[\left(u\left(x\right)-\pi_{h} u\left(x\right)\right)^{'}\right]^{2}dx+\int_{\alpha}^{x_{j+1}}\left[\left(u\left(x\right)-\pi_{h} u\left(x\right)\right)^{'}\right]^{2}dx\right)^{\frac{1}{2}}\\
&&~~~\qquad~~~~\qquad\quad~~\leq c h\left\|u\right\|_{2,I}.
\end{eqnarray*}

Combining all above to get,
\begin{eqnarray*}
&&\left\|u\left(x\right)-\pi_{h} u\left(x\right)\right\|_{1, I}=\left(\left\|u\left(x\right)-\pi_{h} u\left(x\right)\right\|_{0, I}^{2}+\left|u(x)-\pi_{h} u(x)\right|_{1, I}^{2}\right)^{\frac{1}{2}}\\
&&~~~\quad\qquad\qquad\quad~~~~\leq c h\left\|u\right\|_{2, I},
\end{eqnarray*}
which completes the proof.
\end{proof}

The following theorem states that the IFEM in 1D provides optimal convergence as that the FEM for regular problems.
\begin{theorem}{}
If $u\left(x\right)$ is the solution of $\left(\ref{sol ch2}\right)$,
and $\pi_{h}u\left(x\right)$ is the interpolating function defined in $\left(\ref{int ch2}\right)$, then
\begin{equation}\label{error ch2}
\left\|u-\pi_{h} u\right\|_{1}\leq ch\left\|u\right\|_{2},
\end{equation}
\end{theorem}
\begin{proof}
We would using $\left(\ref{pp ch2}\right)$ and $\left(\ref{theo ch2}\right)$ to get
\begin{eqnarray*}
&&\left\|u\left(x\right)-\pi_{h} u\left(x\right)\right\|_{1}=\left(\int_{0}^{1}\left(u\left(x\right)-\pi_{h} u\left(x\right)\right)^{2}+\left(\left(u\left(x\right)-\pi_{h} u\left(x\right)\right)^{'}\right)^{2}dx\right)^{\frac{1}{2}}\\
&&\qquad\qquad\qquad~~~~~=\left(\sum_{E_{i}}\left\|u-\pi_{h} u\right\|_{1, E_{i}}^{2}+\left\|u-\pi_{h} u\right\|_{1, I}^{2}\right)^{\frac{1}{2}}\\
&&\qquad\qquad\qquad~~~~~\leq \left(\sum_{E_{i}} ch^{2}\left\|u\right\|_{2, E_{i}}^{2}+ch^{2}\left\|u\right\|_{2, I}^{2}\right)^{\frac{1}{2}}\\
&&\qquad\qquad\qquad~~~~~\leq ch\left(\sum_{E_{i}}\left\|u\right\|_{2, E_{i}}^{2}+\left\|u\right\|_{2, I}^{2}\right)^{\frac{1}{2}} \\
&&~~~\qquad\qquad\qquad~~\leq ch\left\|u\right\|_{2}.
\end{eqnarray*}
This completes the proof of the theorem.
\end{proof}

\subsection{An accurate flux computation at the left of the interface}

In this sub-section, we explain how to get an accurate flux or
first order derivative of the solution at the interface from the left side of the interface.
The method is based on the approach proposed in \cite{wheeler1974galerkin}
for flux computations at boundaries.
The method is based on the use of the Galerkin solution of the problem
and it is different from other posterior error analysis.

We define the following  $\Gamma_{\alpha}^{-}$
\begin{equation}\label{f ch2}
  \Gamma_{\alpha}^{-}\triangleq\frac{1}{\alpha}\{(\beta u_{h}^{'}, 1)_{(0,\alpha)}
  +(qu_{h}-f, x)_{(0,\alpha)}\},
\end{equation}
as an approximation to the exact flux $\beta_{1}u_{x}^{-}\left(\alpha\right)$.
Below we show that this is a second order approximation which improves the accuracy of the flux by one order
 compared the estimate in (\ref{error ch2}).
\begin{theorem}{}
If $u\left(x\right)$ is the solution of $\left(\ref{sol ch2}\right)$,
$u_{h}\left(x\right)$ is the Galerkin approximation of the solution of $u\left(x\right)$,
$\Gamma_{\alpha}^{-}$ is as defined above, then
\begin{equation}\label{a ch2}
\left|\Gamma_{\alpha}^{-}-\beta_{1}u_{x}^{-}\left(\alpha\right)\right|\leq ch^{2}\left\|u\right\|_{2}.
\end{equation}
\end{theorem}
\begin{proof}
We define $Y \in V_{h}$ as a function that satisfied $Y(0)=0$ and
\begin{equation}\label{g ch2}
\left(\beta Y^{'} , v_{h}^{'}\right)_{\left(0,\alpha\right)}+\left(qY-f , v_{h}\right)_{\left(0,\alpha\right)}=\beta_{1}u_{x}^{-}\left(\alpha\right)v_{h}\left(\alpha\right) , \qquad  \forall v_{h}\in V_{h}, \ and ~v_{h}\left(0\right)=0.
\end{equation}
\ Subtracting (\ref{f ch2}) with $v_{h}=x/\alpha$~from (\ref{g ch2}), we have
\begin{eqnarray*}
&&\left|\Gamma_{\alpha}^{-}-\beta_{1}u_{x}^{-}\left(\alpha\right)\right|=\frac{1}{\alpha}\left|\left(\beta\left(u_{h}-Y\right)^{'}, 1 \right)_{\left(0,\alpha\right)}+\left(q\left(u_{h}-Y\right), x\right)_{\left(0,\alpha\right)}\right|\\
&&~~~\qquad\qquad\quad~~~\leq c\left\{\left\|u_{h}-Y\right\|_{0}+\left|\left(u_{h}-Y\right)\left(\alpha\right)\right|\right\}.
\end{eqnarray*}
From (\ref{c ch2}) and (\ref{g ch2}) we can see that
\begin{equation}\label{}
\left(\beta \left(u_{h}-Y\right)^{'} , v_{h}^{'}\right)_{\left(0,\alpha\right)}+\left(q\left(u_{h}-Y\right) , v_{h} \right)_{\left(0,\alpha\right)}=0 , \qquad \forall v_{h} \in V_{h,\left(0,\alpha\right)};
\end{equation}
Set $w=u_{h}-Y$, and $v_{h}=w-xw(\alpha)$, we get the following equation
\begin{equation*}\label{}
\left(\beta w^{'} , w^{'}\right)_{\left(0,\alpha\right)}+\left(qw , w\right)_{\left(0,\alpha\right)}=\left(\beta w^{'} , w\left(\alpha\right)\right)_{\left(0,\alpha\right)}+\left(qw , xw\left(\alpha\right)\right)_{\left(0,\alpha\right)},
\end{equation*}
and thus we have
\begin{equation}\label{h ch2}
\left\|w\right\|_{1}\leq c\left\{\left\|w\right\|_{0}+\left|w\left(\alpha\right)\right|\right\}.
\end{equation}

Next we construct the following auxiliary problem.
Let $\varphi \in H^{2}(0, \alpha)\cap \widetilde{H}^{1}\left(0, \alpha\right)$ be the solution of
the following
\begin{equation*}\label{}
\begin{cases}
 L^{*}\varphi =-w ,\quad   w\in \left(0 , \alpha\right),\\
\varphi\left(0\right)=\varphi\left(\alpha\right)=0.
 \end{cases}
\end{equation*}
 We also assume that
\begin{equation*}\label{}
\left\|\varphi\right\|_{2}\leq  c \left\|w\right\|_{0}.
\end{equation*}
Then, for an appropriately chosen $\pi\varphi \in V_{h,\left(0,\alpha\right)}$, we proceed to get the following,
\begin{eqnarray*}
&&\left(w , w\right)_{\left(0,\alpha\right)}=\left|-\left(w , L^{*}\varphi\right)_{\left(0,\alpha\right)}\right|=\left|\left(w , \left(\beta\varphi^{'}\right)^{'}-q\varphi\right)_{\left(0,\alpha\right)}\right|\\
&&\qquad~~~~~~~~=\left|-\left(\beta w^{'} , \varphi^{'}\right)_{\left(0,\alpha\right)}-\left(qw , \varphi \right)_{\left(0,\alpha\right)}+\left(\beta ^{-}\varphi^{'}w\right)\left(\alpha\right)\right|\\
&&\qquad\quad~~~~~\leq\left|\left(\beta w^{'} , \varphi^{'}-\left(\pi\varphi\right)^{'}\right)_{\left(0,\alpha\right)}\right|+\left|\left(qw , \varphi-x\right)_{\left(0,\alpha\right)}\right|+\left|\left(\beta\varphi^{'} w\right)\left(\alpha\right)\right|\\
&&\qquad\quad~~~~~\leq c \left\{\left\|w\right\|_{1}\left\|\varphi-\pi\varphi\right\|_{1}+\left|\varphi^{'}\left(\alpha\right)\right|\left|w\left(\alpha\right)\right|\right\}\\
&&\qquad\quad~~~~~\leq c \left\{h\left\|w\right\|_{1}\left\|\varphi\right\|_{2}+\left\|w\right\|_{0}\left|w\left(\alpha\right)\right|\right\}\\
&&\qquad\quad~~~~~\leq c\left\|w\right\|_{0}\left\{h\left\|w\right\|_{1}+\left|w\left(\alpha\right)\right|\right\}.
\end{eqnarray*}
The above yields,
\begin{equation}\label{i ch2}
\left\|w\right\|_{0}\leq c \left\{h\left\|w\right\|_{1}+\left|w\left(\alpha\right)\right|\right\}.
\end{equation}
For $h$ sufficiently small (\ref{h ch2}) and (\ref{i ch2}) imply that
\begin{equation}\label{}
\left\|w\right\|_{0}\leq c \left|w\left(\alpha\right)\right|,
\end{equation}
where $c$ is a positive constant depending only on the coefficients $\beta$, $q(x)$.

We now derive an estimate of $\left|w\left(\alpha\right)\right|$, using the new auxiliary problem
\begin{equation*}\label{}
\begin{cases}
 L^{*}\xi =0 ,\quad   \xi\in \left(0 , \alpha\right),\\
\xi\left(0\right)=0 ,\quad ~\beta_{1}\left(\alpha\right)\xi^{'}\left(\alpha\right)=1 .
 \end{cases}
\end{equation*}
Let $\eta=u-Y$, then we can get
\begin{equation}\label{}
\left(\beta\eta^{'} , v_{h}\right)_{\left(0,\alpha\right)}+\left(q\eta , v_{h}\right)_{\left(0,\alpha\right)}=0 , \quad~\forall v_{h}\in V_{h}~~ such~that ~~v_{h}\left(0\right)=0.
\end{equation}
Furthermore, using
\begin{eqnarray*}
&&0=-\left(\eta , L^{*}\xi\right)_{\left(0,\alpha\right)}=-\left(\eta , -\left(\beta\xi^{'}\right)^{'}+q\xi\right)_{\left(0,\alpha\right)}\\
&&\qquad\quad\qquad\qquad~~~=\left(\eta , \left(\beta\xi^{'}\right)^{'}-q\xi\right)_{\left(0,\alpha\right)}\\
&&\qquad\quad\qquad\qquad~~~=\left(-\beta\eta^{'} , \xi^{'}\right)_{\left(0,\alpha\right)}+\left(-q\eta , \xi \right)_{\left(0,\alpha\right)}+\eta\left(\alpha\right),
\end{eqnarray*}
we have
\begin{eqnarray*}
&&\left|\eta\left(\alpha\right)\right|=\left|\left(\beta\eta^{'} , \xi^{'}\right)_{\left(0,\alpha\right)}+\left(q\eta , \xi\right)_{\left(0,\alpha\right)}\right|\\
&&\quad~~~~~\leq c\left\|\eta\right\|_{1}\left\|\xi-x\right\|_{1}\\
&&\quad~~~~~\leq ch\left\|\eta\right\|_{1}\\
&&\quad~~~~~\leq ch^{2}\left\|u\right\|_{2}.
\end{eqnarray*}
Finally, we get
\begin{eqnarray*}
&&\left|\Gamma_{\alpha}^{-}-\beta_{1}u_{x}^{-}\left(\alpha\right)\right|\leq c\left|\left(u_{h}-Y\right)\left(\alpha\right)\right|\\
&&~~~~\quad\qquad\quad~~~~\leq c\left\{\left|\left(u-u_{h}\right)\left(\alpha\right)\right|+\left|\left(u-Y\right)\left(\alpha\right)\right|\right\}\\
&&~~~~\quad\qquad\quad~~~~\leq ch^{2}\left\|u\right\|_{2}.
\end{eqnarray*}
This completes the proof of the theorem.
\end{proof}

\subsubsection*{Approximation of flux from the right side of the interface}

In the similar way, we can get the second order accurate flux, $-\beta^+ u_x^+ $, from the right side of the interface
\begin{equation*}
 \Gamma_{\alpha}^{+}\triangleq\frac{1}{1-\alpha}\{(\beta u_{h}^{'}, -1)_{(\alpha,1)}+(qu_{h}-f, 1-x)_{(\alpha,1)}\}.
\end{equation*}
We also have the following error bound.
\begin{equation*}\label{}
  \left|\Gamma_{\alpha}^{+}+\beta^+ u_{x}^{+}\left(\alpha\right)\right|\leq ch^{2}\left\|u\right\|_{2}.
\end{equation*}

\subsubsection*{Approximation of fluxes at the boundary}

The approach for accurate flux computations at the interface
can be applied to the flux computation from the left and right boundaries as expressed below.
We define approximations $\Gamma_{0}$ and $\Gamma_{1}$ to the fluxes,
$\beta_{1}u^{'}\left(0\right)$ and $\beta_{2}u^{'}\left(1\right)$ respectively:
\begin{equation*}\label{}
\Gamma_{0} \triangleq\left(\beta u_{h}^{'}, -1\right)+\left(qu_{h}-f, 1-x\right),
\end{equation*}
\begin{equation*}\label{}
\Gamma_{1}\triangleq\left(\beta u_{h}^{'}, 1\right)+\left(qu_{h}-f, x\right).
\end{equation*}
Then $\Gamma_{0}$ and $\Gamma_{1}$ are second order approximations to the flux
from the left and right boundaries as stated in the following theorem.
We skip the proof since it is similar to that for the flux from the left side of the interface.
\begin{theorem}{}
If $u\left(x\right)$ is the solution of $\left(\ref{sol ch2}\right)$,
$u_{h}\left(x\right)$ is the Galerkin approximation of the solution of $u\left(x\right)$,
$\Gamma_{0}$ and $\Gamma_{1}$ are as defined above,
then
\begin{equation*}\label{}
\left|\Gamma_{0}+\beta_{1}u^{'}\left(0\right)\right|+\left|\Gamma_{1}-\beta_{2}u^{'}\left(1\right)\right|\leq c h^{2}\left\|u\right\|_{2}.
\end{equation*}
\end{theorem}

\subsection{Numerical experiments in 1D}

We present one example below that is taken from \cite{li2006immersed}.
The exact solution is
\begin{equation*}
  u(x)=
  \begin{cases}
    x^4/\beta^-, \qquad \quad  \textrm{if}~~ 0<x < \alpha,\\
    x^4/\beta^+ +(1/\beta^--1/\beta^+)\alpha^4, \; \;\textrm{if}~~ \; \alpha < x <1,
 \end{cases}
\end{equation*}
where $0<\alpha<1$ is an interface.
The solution satisfies the ODE $-(\beta u')'=f(x)$ where $f(x)=-12 x^2$.
In this example, $q(x)=0$ and $[u]=0$ and $[\beta u']=0$, but $[u']\neq~0$.

In Table~\ref{result1D ch2},
we show a grid refinement analysis for the proposed method with $\alpha=1/3$, $\beta^-=2$, $\beta^+=10$.
Thus the interface $\alpha$ is not a nodal point.
We measure the  error for the solution $u(x)$is the entire domain $(0,1)$
in the second column using the strongest norm $L^{\infty}$ norm.
We estimate the convergence order using $p = \log (E_n/E_{2N})/\log 2$ in the third column.
As usual, since the relative location between the underlying grid and the interface $\alpha$ is not fixed,
the convergence order fluctuates.
The average convergence order is $1.983$.
In the third column, we list the grid refinement analysis for $u_x^-=\lim_{x\rightarrow \alpha, x<\alpha} u'(x)$,
that is, the first order derivative from the left side of the interface,
we observe clear second order convergence as shown in the fifth column.

\begin{table}[H]
\centering
\begin{tabular}{|c|c|c|c|c|}
\hline
  N & $\|u-u_h\|_{L^{\infty}}$ &  Order   & $|u-u_h|_{H^1}$  & Order  \\
\hline
16  & 3.395E-05  &       & 3.870E-03 &          \\
\hline
32  & 1.547E-05  & 1.134 & 7.980E-04 & 2.278     \\
\hline
64  & 2.191E-06  & 2.820 & 1.562E-04 & 2.353     \\
\hline
128 & 9.732E-07  & 1.171 & 3.892E-05 & 2.005     \\
\hline
256 & 1.413E-07  & 2.784 & 8.475E-06 & 2.199     \\
\hline
512 & 6.088E-08  & 1.215 & 2.263E-06 & 1.905     \\
\hline
1024& 8.900E-09  & 2.774 & 5.098E-07 & 2.150     \\
\hline
\end{tabular}
\caption{ A grid refinement analysis  of the proposed method with $\alpha=1/3$, $\beta^-=2$, $\beta^+=10$. The second column is the $L^{\infty}$ error of the solution in the entire domain $(0,1)$. The fourth column is the error in the first order derivative $u_x^-$, that is, from the left side of the interface. The average convergence order for the solution and $u_x^-$ are $1.983$ and $2.148$, respectively.}  \label{result1D ch2}
\end{table}

\section{The numerical method and experiments for the 2D interface problem}

The results in the previous section are the optimal since both the solution and the flux (at the interface or boundaries)
using a piecewise linear finite element space.
However, it is still an open question how to apply the approach to 2D problems with a curved interface.
In this section, we provide an alternative approach that is similar to a mixed finite element method
but only in a small tube around the interface.

In this section, the elliptic interface problem is
\begin{equation}\label{2Dpde ch2}
-\nabla \cdot \left(\beta \nabla u(\mathbf{x}) \right) +  q(\mathbf{x}) u(\mathbf{x})  = f(\mathbf{x}), \quad \mathbf{x}\in \Omega_i, \quad \mbox{$i=1,2$},
\end{equation}
where $q(\mathbf{x})\in L^{\infty}(\Omega) \ge 0$, $\Omega=\Omega^-\cup \Omega^+$,
$\beta(\mathbf{x})$ is a piecewise positive constant as in (\ref{beta ch2})
and has a finite jump discontinuity across a closed interface $\Gamma\in C^1$ in the solution domain.

In our new method, we introduce a tube  that contains the interface with a diameter $2 \epsilon$ as
\begin{equation*}
  \Omega_{\epsilon} = \left \{ \mathbf{x}\in \Omega_1 \cup \Omega_2, \quad d(\mathbf{x},\Gamma ) \le \epsilon \right\},
\end{equation*}
where $d(\mathbf{x},\Gamma )$ is the distance between $\mathbf{x}$ and the interface~$\Gamma$.
In the tube~$\Omega_{\epsilon}$,
we introduce the flux as a separate variable vector $\mathbf{v}$
that can be considered as an augmented variable.
Thus, in addition to the PDE (\ref{2Dpde ch2}) in the entire domain,
we also have the following equations

\begin{equation*}
\begin{array}{c}
   -\beta \, \nabla u = \mathbf{v},    \\
    \nabla \cdot \mathbf{v} + q u = f ,
 \end{array}
\qquad \mathbf{x} \in \Omega_{\epsilon}.
\end{equation*}

Next we  define the following functional spaces
\begin{eqnarray*}
H_0^1 &=& \left\{\phi \in H^1(\Omega), \quad \phi = 0 \text{ on } \partial\Omega  \right\},\\
W & =& \left\{ w\in L^2(\Omega_{\epsilon}) \right\}, \\
L_g &=& \left \{\mathbf{g} \in (L^2(\Omega_{\epsilon}))^2 \right\},
\end{eqnarray*}
assuming homogenous boundary condition along $\partial\Omega_2$.

We can easily get the following weak form for $u$ (in the entire domain) and $\mathbf{g}$ (in $\Omega_{\epsilon}$) below
\begin{eqnarray}
 &   \left(\beta_i \nabla u, \, \nabla \phi \right) + \left(q u, \phi\right) = \left(f, \phi\right)  \text{ in } \Omega_i, \;  i=1, 2, \\
 &  -\left(\beta_i \nabla u, \, \mathbf{g}\right) = \left(\mathbf{v}, \mathbf{g}\right) \;  \text{in} \; \Omega_{\epsilon}\cap \Omega_{i}, \; i=1,2, \label{poi2}\\
 &  \left(\nabla \cdot \mathbf{v},\, w\right) + (q u, w) = \left(f, w \right) \quad \text{ in } \Omega_{\epsilon}, \label{poi3}
\end{eqnarray}
where the inner product is in the regular $L^2$ sense and those quantities  $\phi$, $\mathbf{g}$, and $w$ are from the spaces defined above.

 There are two intuitive reasonings behind the new algorithm.
 We know that the mixed formulation would improve the gradient computation.
 If we are only interested in the gradient from each side of the interface,
 then we just need to use a small tube for the computation.
 The second consideration is that if we set the flux
 $\mathbf{v}=\beta \nabla u$ along the interface
 as an unknown in addition to the solution $u$,
 and then discretize the whole system with high order discretization (second order in the manuscript),
 then we would expect the error for the unknown flux $\mathbf{v}$ would have the same order of accuracy as for the discretization.
In discretization, we use piecewise linear functions for $\phi$ and $\mathbf{g}$ as usual,
and piecewise constant functions for $w$.
The new augmented method enlarged the system by (\ref{poi2}) and (\ref{poi3}).
In terms of the stiffness matrix,
an additional $n_v$ number of columns are added to the stiffness matrix
where $n_v$ is the number of extra unknowns  $\mathbf{v}$ in the tube $\Omega_{\epsilon}$.
As a result, the stiffness matrix becomes a rectangular matrix instead of a square matrix.
We used the singular value decomposition (SVD) to solve the resultant rectangular system.
Since $\mathbf{v}$ has co-dimension one compared with that of $u$,
the additional extra cost is negligible compared with that of the elliptic solver on the entire domain.

\subsection{Numerical experiments in 2D}

Let $\Omega_2$ be a unit circle centered at $(0,0)$ with radius $R=1$.
In our numerical test, we take $q=0$.
Let $\Gamma$ be an interface inside the unit circle with radius $R=0.9$.
The tube width is taken as $\epsilon=3h$, that is three layers from each side of the interface.
The  exact solution is
\begin{eqnarray}
 u(x,y)  = \sin x \cos y,
\end{eqnarray}
in the entire domain so that the solution is continuous,
but the flux $\beta \nabla u$ is discontinuous for the test problem.
The coefficient is taken as $\beta_1=100$ and $\beta_2=1$.
The source terms and boundary condition are determined accordingly.
In Table $\ref{table41 ch2}$,
the $L^2$ norm errors of $u$, $\mathbf{v}$,
and the $H^1$ norm error of $u$ are reported.
The $L^2(\Omega)$, $H^1(\Omega)$ are used for the solution $u$,
that is, in the entire domain, while $L^2(\Gamma)$ is used for the flux along the interface.

The first column $N$ is the mesh lines in the coordinates directions.
The results indicate that the new augmented method worked as expected.
The convergence rate is shown in Table $\ref{table42 ch2}$.

\begin{table}[H]
\centering
\begin{tabular}{|c|c|c|c|}
\hline
 $N$  & $L^2$ error of ${u}$ & $H^1$ error of ${u}$ & $L^2$ error of $\mathbf{v}$ \\
\hline
8 & 9.96e-3 & 5.67e-1 & 5.67e-1 \\
\hline
16 & 2.48e-3 & 1.75e-1 & 1.75e-1 \\
\hline
32 & 7.77e-4 & 4.14e-2 & 4.14e-2 \\
\hline
64 & 1.81e-4 & 1.04e-2 & 1.04e-2 \\
\hline
128 & 4.71e-5 & 5.95e-3 & 5.95e-3 \\
\hline
256 & 1.15e-5 &1.54e-3 & 1.54e-3 \\
\hline
512 & 2.92e-6 &3.80e-4 & 3.80e-4 \\
\hline
\end{tabular}
\caption{A grid refinement analysis of the proposed method.  The $H^1$ norm of $u$ is the same as the $L^2$ error of $\mathbf{v}$ because the error in the gradient is dominated compared with that of $u$.} \label{table41 ch2}
\end{table}

In Table $\ref{table42 ch2}$,
a comparison of the convergence order between the standard finite element method and the new augmented approach is presented.
We observe that the new approach provides much better accuracy for the flux (gradient).
Now we have super-convergence for the gradient.
Super-convergence here is the result of the convergence that is faster than the original method.
For the finite element method with the piecewise linear function space,
it is well-known that the flux is first order accurate in the $L^2$ norm.
In our manuscript, We proposed a method to reconstruct the flux at the interface
and has shown that the reconstructed flux  has second order convergence.
\begin{table}[H]
\centering
\begin{tabular}{|c|c|c|c|c|}
\hline
Quantity &$u$  & $u$ &$\mathbf{v}$ \\
\hline
 Norm & $L^2$ & $H^1$ &$L^2$ \\
\hline
Order (standard FEM) & 1.98 & 1.03 & 1.03\\
\hline
Order (new  method)  & 1.94 & 1.72 & 1.72\\
\hline
\end{tabular}
\caption{A comparison of the convergence order between the standard FEM and the new augmented approach.\label{table42 ch2}}
\end{table}

Next, in Table~\ref{table43 ch2} we present the results with different interface locations
including the case that covers the entire domain ($r_i=0$)
so that we get the gradient in the entire domain as well,
where $r_i$ is the radius of the interface.
Of course, the computational cost also increased.
As we expected, we have second order convergence in the $L^2$ norm for the solution,
and roughly $1.70$ order for the flux (gradient).
In this case, the accuracy of the gradient  is improved by about $70$ percent.

\begin{table}[H]
\centering
\begin{tabular}{|c|c|c|c|c|}
\hline
$r_i$ &Order in $L^2$ of $u$  & order in $H^1$  \\
\hline
 $0.9$  & 1.94 & 1.72\\
\hline
 $0.99$  & 1.94 & 1.70\\
\hline
 0  & 1.93 & 1.70\\
\hline
\end{tabular}
\caption{A comparison of the convergence order for various location of the interface, where $r_i$ is the radius of the interface.\label{table43 ch2}}
\end{table}

Now we show the result with a non-zero $q(x,y)$.
We use the same solution above with $q(x,y)=1$.
The source term is modified accordingly.
In Table~\ref{table41_new_q},
we show the grid refinement analysis of the error in $L^2$ and $H^1$ norm.
We observe the same behavior with the similar convergence orders.
The average convergence orders are $1.92$ for the $L^2$ norm, $1.71$ for the $H^1$ norm, respectively.
Note that, the $L^2$ norm of the error in $\mathbf{v}$ is the same as the $H^1$ norm as explained earlier.

\begin{table}[H]
\centering
\begin{tabular}{|c|c|c|}
\hline
 $N$  & $L^2$ error of ${u}$ & $H^1$ error of ${u}$ \\
\hline
8 & 9.96e-3 & 5.67e-1  \\
\hline
16 & 3.51e-3 & 2.61e-1  \\
\hline
32 & 1.17e-4 & 6.54e-2  \\
\hline
64 & 2.81e-4 & 1.66e-2  \\
\hline
128 & 7.24e-5 & 9.13e-3  \\
\hline
256 & 1.85e-5 &2.31e-3  \\
\hline
512 & 4.62e-6 &5.79e-4  \\
\hline
\end{tabular}
\caption{A grid refinement analysis of the proposed method when $q(x,y)=1$. The average convergence orders are $1.92$ for the $L^2$ norm, $1.71$ for the $H^1$ norm, respectively.} \label{table41_new_q}
\end{table}

In the previous example, the solution is the same in the entire domain in spite of the flux is discontinuous.
Below we present another example in which the solution is different in different domain.
The outer boundary is $R=2$.
\begin{equation}\label{}
u(x,y,t) =
\begin{cases}
 (x^2 + y^2)^2   & \mbox{if $ r > 1, $}\\
 (x^2 + y^2)     & \mbox{if $ r \leq 1, $}
 \end{cases}
\end{equation}
where $r=\sqrt{x^2 + y^2}$.
The source term $f(x,y)$,
and the Dirichlet boundary condition are  determined from the true solution.
In this example, the solution is continuous, that is, $[u]=0$, but the flux jump is non-homogeneous.

We tested our new method with large jump ratios $\beta_2 : \beta_1=1000:1$ and $\beta_2 : \beta_1=1:1000$.
In Table~\ref{table53 ch2},
we present the results with different widths of the tube
including the case that covers the entire domain
so that we get the gradient in the entire domain as well.
As we expected,
we have second order convergence in the $L^2$ norm for the solution,
and roughly $1.54$ order for the flux (gradient) as in the thin tube case.
Compared with the standard finite element method, the accuracy of the computed gradient is improved by more than $50$ percent.
Note that the results are almost the same as the new gradient recovery technique using a posterior approach \cite{guo-yang16}
in which the rate of the recovered gradient is around $1.5$.
Note that, the convergence order for the gradient is about $1.54$
which is lower than the previous case possibly due to the non-homogenous flux jump.

\begin{table}[H]
\centering
\begin{tabular}{|c|c|c|c|c|}
\hline
width ($\epsilon$) &Order in $L^2$ of $u$  & order in $H^1$  \\
\hline
 $3h$  & 1.96 & 1.53\\
\hline
 $10h$  & 1.96 & 1.56\\
\hline
 2 (entire domain) & 1.96 & 1.56\\
\hline
\end{tabular}
\caption{A comparison of the convergence order for various widths  of the tube.\label{table53 ch2}}
\end{table}

\section{Conclusions}

In this paper, we discussed two methods to enhance the accuracy of the computed flux
at the interface for the elliptic interface problem.
One is for one-dimensional problems
in which we use a simple weak form to get second order accurate fluxes at the interface from each side.
We also have rigorous analysis for the approach.
The other one is an augmented approach for two dimensional interface problems.
Numerical examples show that we get better than super-convergence (about $1.50\sim 1.70$ order)
for the computed fluxes at the interface from each side.
For the two dimensional algorithm, the theoretical analysis is still an open challenge.

\section*{Acknowledgment}
The authors would like to thank the following supports.
F. Qin is partially supported by China NSF Grants 11691209, 11691210
and the NSF of Jiangsu Province, China (Grant No. BK20160880).
Z. Ma is supported by China National Special Fund: 2012DFA60830,
and
Z. Li would is partially supported by the US NSF grant DMS-1522768, and China NSF grant 11371199 and 11471166.


\begin{thebibliography}{99}
\bibitem{adams2003sobolev}
R.~Adams and J.~Fournier.
 \newblock {\em Sobolev spaces. Second edition,}
 \newblock Pure and Applied Mathematics (Amsterdam), 140. Elsevier/Academic Press, Amsterdam, 2003.

\bibitem{an2014partially}
N. An and H. Chen.
\newblock A partially penalty immersed interface finite element method for anisotropic elliptic interface problems.
\newblock {\em  Numerical Methods for Partial Differential Equation}, 30: 1984--2028, 2014.

\bibitem{babuska1970finite}
I.~Babu\v{s}ka.
The finite element method for elliptic equations with  discontinuous coefficients.
\newblock {\em Computing}, 5:  207--213, 1970.

\bibitem{bramble1996finite}
J.~Bramble and J.~King.
A finite element method for interface problems in domains with smooth boundaries and interfaces.
 \newblock {\em  Advances in Computational Mathematics},  6:  109--138, 1996.

\bibitem{brenner2007mathematical}
S.~Brenner and R.~Scott.
\newblock {\em The mathematical theory of finite element methods}, volume~15.
\newblock Springer Science \& Business Media, 2007.

\bibitem{carstensena2014low}
C.~Carstensena, D.~Gallistlb, F.~Hellwiga and L.~ Wegglera.
Low-order d{PG}-{FEM} for an elliptic {PDE}.
\newblock {\em Computers  Mathematics with Applications},
68(11): 1503--1512, 2014

\bibitem{chen1998finite}
Z.~Chen and J.~Zou.
Finite element methods and their convergence for  elliptic and parabolic interface problems.
\newblock {\em Numerische Mathematik}, 79:  175--202, 1998.

\bibitem{chou2012immersed}
S. Chou.
 An immersed linear finite element method with interface flux capturing recovery.
\newblock {\em Discrete and Continuous Dynamical Systems Series B}, 17: 2343--2357, 2012.

\bibitem{chou2010optimal}
S.~Chou, D.~Kwak, and K.~Wee.
Optimal convergence analysis of an immersed interface finite element method.
\newblock {\em  Advances in Computational Mathematics}, 33:  149--168, 2010.

\bibitem{douglas1974galerkin}
J. Douglas, T. Dupont, and M. Wheeler.
A {G}alerkin procedure for approximating the flux on the boundary for elliptic and parabolic boundary value problems.
\newblock {\em RAIRO Model. Math. Anal. Numer.}, 8: 47--59, 1974.

\bibitem{guo-yang16}
H.~Guo \and X.~Yang.
Gradient recovery for elliptic interface problem:  {II}. Immersed finite element methods. ArXiv:1608.00063,  2016.

\bibitem{he2010immersed}
X.~He, T.~Lin, and Y.~Lin.
Immersed finite element methods for elliptic interface problems with non-homogeneous jump conditions.
\newblock {\em   International Journal of Numerical Analysis and Modeling}, 8:  284--301, 2010.

\bibitem{ji2016new}
H. Ji, J. Chen, and Z. Li.
\newblock A new augmented immersed finite element method without using
SVD interpolations.
\newblock {\em Numerical Algorithms}, 71: 395--416, 2016.

\bibitem{kevorkian1990partial}
J.~Kevorkian.
\newblock {\em Partial differential equations}, Wadsworth \& Brooks/Cole,
  1990.

\bibitem{kwak2010analysis}
D.~Kwak, K. Wee, and K. Chang.
\newblock An analysis of a broken ${P}_1$-nonconforming finite element method
  for interface problems.
\newblock {\em SIAM Journal on Numerical Analysis}, 48: 2117--2134, 2010.

\bibitem{li1998immersed}
Z.~Li.
\newblock The immersed interface method using a finite element formulation.
\newblock {\em Applied Numerical Mathematics}, 27: 253--267, 1998.

\bibitem{li2006immersed}
Z.~Li and K.~Ito.
\newblock {\em The immersed interface method: numerical solutions of PDEs
  involving interfaces and irregular domains}, volume~33.
\newblock {SIAM}, 2006.

\bibitem{li2003new}
Z.~Li, T.~Lin, and X.~Wu.
\newblock New {C}artesian grid methods for interface problems using the finite
  element formulation.
\newblock {\em Numerische Mathematik}, 96: 61--98, 2003.

\bibitem{lin2015partially}
 T. Lin, Y. Lin and X. Zhang.
\newblock Partially penalized immersed finite element methods for elliptic interface problems.
\newblock {\em SIAM Journal on Numerical Analysis}, 53: 1121--1144, 2015.

\bibitem{lin2012linear}
T.~Lin and X.~Zhang.
Linear and bilinear immersed finite elements for  planar elasticity interface problems.
\newblock {\em  Journal of Computational and Applied Mathematics},  236: 4681--4699, 2012.

\bibitem{sutton1995interfaces}
A.~Sutton and R.~Balluffi.
\newblock {\em Interfaces in crystalline materials}.
\newblock Clarendon Press, 1995.


\bibitem{tartar2007introduction}
L.~Tartar.
\newblock {\em An introduction to Sobolev spaces and interpolation spaces}.
\newblock Lecture Notes of the Unione Matematica Italiana, 3. Springer, Berlin; UMI, Bologna, 2007.

\bibitem{wahlbin1995superconvergence}
L. Wahlbin.
\newblock {\em Superconvergence in {G}alerkin finite element methods},
  Lecture Notes in Mathematics, vol. 1605, Springer-Verlag, Berlin, 1995.

\bibitem{wheeler1974galerkin}
M. Wheeler. A {G}alerkin procedure for estimating the flux for two-point boundary value problems.
\newblock {\em SIAM Journal on Numerical Analysis}, 11:  764--768, 1974.

\bibitem{xu1982error}
J.~Xu.
Error estimates of the finite element method for the 2nd order elliptic equations with discontinuous coefficients.
\newblock {\em Journal of Xiangtan University }, 1: 1--5, 1982.

\bibitem{yang2003immersed}
X.~Yang, B.~Li, and Z.~Li.
The immersed interface method for elasticity problems with interface.
\newblock {\em Dynamics of Continuous, Discrete and Impulsive Systems}, 10:  783--808, 2003.

\bibitem{zhang2005new}
Z. Zhang and A. Naga.
A new finite element gradient recovery  method: superconvergence property.
\newblock {\em  SIAM  Journal on Scientific Computing}, 26: 1192--1213, 2005.

\bibitem{zienkiewicz2000finite}
O.~Zienkiewicz and R.~Taylor.
\newblock {\em The finite element method: solid mechanics}, volume~2.
\newblock Butterworth-{H}einemann, 2000.
\end{thebibliography}
\end{document}